\documentclass{amsart}
\usepackage{amssymb, mathrsfs, tikz}

\subjclass[2010]{16W50}
\keywords{Group Gradings, Incidence algebras, Associative algebras}

\title{Group gradings on finite dimensional incidence algebras}
\author{Ednei A. Santulo Jr}
\address{Department of Mathematics, State University of Maring\'a, 5790 Colombo Avenue, 87020-900 Maring\'a, PR, Brazil}
\email{easjunior@uem.br}

\author{Jonathan P. Souza}
\thanks{The second author was financed by the Coordena\c c\~ao de Aperfei\c coamento de Pessoal de N\'ivel Superior - Brasil (CAPES) - Finance Code 001}
\address{Department of Mathematics, State University of Maring\'a, 5790 Colombo Avenue, 87020-900 Maring\'a, PR, Brazil}
\email{jonathanprass@hotmail.com}

\author{Felipe Y. Yasumura}
\thanks{The third author was supported by S\~ao Paulo Research Foundation (Fapesp), grant number 2013/22.802-1, and by the Coordena\c c\~ao de Aperfei\c coamento de Pessoal de N\'ivel Superior - Brasil (CAPES) - Finance Code 001}
\address{Department of Mathematics, State University of Maring\'a, 5790 Colombo Avenue, 87020-900 Maring\'a, PR, Brazil}
\email{felipeyukihide@gmail.com}

\newtheorem{Thm}{Theorem}
\newtheorem{Lemma}[Thm]{Lemma}
\newtheorem{Cor}[Thm]{Corollary}
\newtheorem{Prop}[Thm]{Proposition}
\newtheorem*{Question}{Question}

\theoremstyle{definition}
\newtheorem{Def}[Thm]{Definition}

\theoremstyle{remark}
\newtheorem{Remark}[Thm]{Remark}
\newtheorem{Example}{Example}

\newcounter{examples}

\begin{document}
\begin{abstract}
In this work, we classify the group gradings on finite-dimensional incidence algebras over a field, where the field has characteristic zero, or the characteristic is greater than the dimension of the algebra, or the grading group is abelian.

Moreover, we investigate the structure of $G$-graded $(D_1,D_2)$-bimodules, where $G$ is an abelian group, and $D_1$ and $D_2$ are the group algebra of finite subgroups of $G$. As a consequence, we can provide a more profound structure result concerning the group gradings on the incidence algebras, and we can classify their isomorphism classes of group gradings.
\end{abstract}
\maketitle

\section{Introduction\label{S_Introduction}}

When one deals with an algebraic structure, it is usually useful to consider an additional structure on it, when possible. Considering group gradings in algebras has proved to be useful to solve relevant mathematical problems more than once. That is the case of the classification of finite-dimensional semisimple Lie algebras, which are graded by the root system (see e.g. \cite{Jac1979}, or \cite{SanMartin}). That is also the case of the positive solution for the Specht problem given by Kemer in \cite{Kemer}. Furthermore, group gradings naturally appear in several branches of Mathematics and Physics. 

We briefly recall the definition of graded algebras. Let $A$ be any algebra (not necessarily associative, not necessarily with unit), and let $G$ be any group. We say that $A$ is $G$-graded if there exists a vector space decomposition $A=\bigoplus_{g\in G}A_g$, where $A_g$ are possibly zero subspaces, such that $A_gA_h\subseteq A_{gh}$, for all $g,h\in G$. The elements of $\cup_{g\in G}A_g$ are called homogeneous, and a non-zero $x\in A_g$ is said to have degree $g$, denoted by $\deg_G x=g$, or simply $\deg x=g$ if there is no ambiguity. An extensive theory concerning graded algebras can be found in the monograph \cite{EldKoc}.

An interesting question is, given an algebra, determine all the possible group gradings on it up to graded isomorphisms. In this direction, the works \cite{BaSeZa2001,BaZa2002} gave the answer for the matrix algebras (they investigate a more general situation), and then, in \cite{bashza}, the authors show how to obtain group gradings on some simple Lie and Jordan algebras from the knowledge of the group gradings on matrix algebras. These works started an extensive research in the subject. The monograph \cite{EldKoc} is a complete state-of-art of the theory.

In this paper, we are interested in the non-simple associative algebras called incidence algebras. These are a generalization for the upper triangular matrices $UT_n$. In short, an incidence algebra is a subalgebra of $UT_n$, generated by matrix units and containing all diagonal matrices (see a precise definition below). Thus, the classification of group gradings on upper triangular matrices is a particular result of the classification of group gradings on the incidence algebras. The classification of group gradings on the upper triangular matrices was obtained in two works \cite{VaZa2007,VinKoVa2004}. To state the classification, we give some definitions. We call a grading on $UT_n$ \emph{good} if all matrix units $e_{ij}$ are homogeneous in the grading; and we call a $G$-grading on $UT_n$ \emph{elementary} if there exists a sequence $(g_1,g_2,\ldots,g_n)\in G^n$ such that every $e_{ij}$ is homogeneous of degree $g_ig_j^{-1}$. Clearly every elementary grading is a good grading, and, in the context of upper triangular matrices, the converse holds. Hence, good and elementary gradings are equivalent notions for the algebra of upper triangular matrices. In \cite{VaZa2007}, it is proved that every group grading on $UT_n$ is elementary, up to an isomorphism. In \cite{VinKoVa2004}, the isomorphism classes of elementary gradings on $UT_n$ are classified, thus obtaining a complete classification of group gradings on this algebra.

Note that the notions of elementary gradings and good gradings can be defined for subalgebras of matrix algebras generated by matrix units. Hence, in particular, one can define these notions for the incidence algebras. It is worth mentioning that these notions were generalized in \cite{BFS2019} for a larger class of algebras. It is natural to conjecture that these two notions are equivalent, and that every group grading on an incidence algebra is elementary up to an isomorphism. However, these two notions are not always equivalent: not every group grading on an incidence algebra is elementary or good (see below). Thus, incidence algebra turns out to be a more intricate combinatorial object in the point of view of gradings by a group.

In this paper, we classify group gradings on the incidence algebra $I(X)$ over a field $\mathbb{F}$, where $X$ is finite and one of the following conditions hold: the base field has characteristic zero, $\mathrm{char}\,\mathbb{F}>\dim I(X)$, or $G$ is abelian.

Some works were dedicated to study the group gradings on the incidence algebras \cite{MSp2010,Jon2006,Pr2013}, but none of them provide a complete understading of its group gradings. Some different phenomenon happens in this algebra. The first surprising fact is that not every good grading on $I(X)$ is necessarily elementary, even if $X$ is finite (see, for instance, \cite[Example 12]{Jon2006}). Moreover, there exist group gradings on $I(X)$ that are not good (see also \cite[Example 1]{MSp2010}):
\begin{Example}\label{ex1}
Let $X=\{1,2,3\}$ be partially ordered by $1\le2$ and $1\le3$. Then we obtain the algebra
$$
I(X)=\left\{\left(\begin{array}{ccc}\ast&\ast&\ast\\&\ast&0\\&&\ast\end{array}\right)\right\}.
$$
Let $G=\mathbb{Z}_2\times\mathbb{Z}_2$, and define
\begin{align*}
A_{(0,0)}=\mathrm{Span}\{e_{11},e_{22}+e_{33}\},\\
A_{(0,1)}=\mathrm{Span}\{e_{22}-e_{33}\},\\
A_{(1,0)}=\mathrm{Span}\{e_{12}-e_{13}\},\\
A_{(1,1)}=\mathrm{Span}\{e_{12}+e_{13}\}.
\end{align*}
This is a well-defined grading on $I(X)$. Moreover, it is not isomorphic to a good grading, since there are not three orthogonal homogeneous idempotents.
\end{Example}
Thus, the study of group gradings on incidence algebras is more complicated than the case of upper triangular matrices. Recall that a grading by a group is equivalent to the action of the group of characters on the algebra, given that the group is abelian and the base field is algebraically closed of characteristic zero (see, for instance, \cite[Section 1.4]{EldKoc}). The automorphism group of the incidence algebra is known \cite[Theorem 7.3.6]{SpieDon}, and it is more intricate than the automorphism group of the upper triangular matrices (see, for instance, \cite{J1995}). Hence, the automorphism group of $I(X)$ suggests the existence of new gradings, other than the gradings on upper triangular matrices.

Another surprising fact is the existence of group gradings on incidence algebras without any homogeneous multiplicative basis (see example in Section \ref{S_Examples}). We recall that a basis $B$ is multiplicative if for all $u,v\in B$, there exists a scalar $\lambda$ such that $\lambda w=uv$, for some $w\in B$. In contrast, it is clear that any good grading admits a multiplicative homogeneous basis. Thus, every group grading on $UT_n$ admits a multiplicative homogeneous basis.

In the next section, we are going to prove our main result:
\begin{Thm}\label{class}
Let $\mathbb{F}$ be a field, $X$ a finite poset, and let $I(X)$ be endowed with a $G$-grading. Assume at least one of the following conditions: $\mathrm{char}\,\mathbb{F}=0$, $\mathrm{char}\,\mathbb{F}>\dim I(X)$, or $G$ is abelian. Then, up to a graded isomorphism, there exist finite abelian subgroups $H_1,\ldots,H_t\subseteq G$, such that: for each $i=1,\ldots,t$, $\mathrm{char}\,\mathbb{F}$ does not divide $|H_i|$, and $\mathbb{F}$ contains a primitive $\mathrm{exp}\,H_i$-root of $1$; and
$$
I(X)\cong\left(\begin{array}{cccc}\mathbb{F}H_1&M_{1,2}&\ldots&M_{1,t}\\&\mathbb{F}H_2&\ddots&\vdots\\&&\ddots&M_{t-1,t}\\&&&\mathbb{F}H_t\end{array}\right),
$$
where each $M_{i,j}$ is a graded $(\mathbb{F}H_i,\mathbb{F}H_j)$-bimodule.
\end{Thm}

The set $\mathcal{E}=\{1,2,\ldots,t\}$ has a natural structure of poset, and it plays an essential role in establishing the isomorphism classes of group gradings on the incidence algebras. The partial order of $\mathcal{E}$ is given by $i\trianglelefteq j$ if $M_{i,j}\ne0$ (see Remark \ref{remarkassociatedposet}).

Further, if we assume the grading group abelian, then we can prove a stronger statement concerning the bimodules $M_{ij}$:
\begin{Thm}\label{Thm3}
Assume that $G$ is abelian, and, using the notation of Theorem \ref{class}, fix $i<j$ and let $\mathcal{H}_{ij}=H_i\cap H_j$. Then, there exist pairwise distinct characters $\chi^{(ij)}_1,\ldots,\chi^{(ij)}_{s_{ij}}\in\hat{\mathcal{H}}_{ij}$, and $m_1,\ldots,m_{s_{ij}}\in M_{ij}$ homogeneous such that
$$
M_{ij}=\mathbb{F}H_im_1\mathbb{F}H_j\oplus\cdots\oplus\mathbb{F}H_im_{s_{ij}}\mathbb{F}H_j,
$$
where $hm_\ell=\chi^{(ij)}_\ell(h)m_\ell h$, for all $h\in\mathcal{H}_{ij}$, for $\ell=1,2,\ldots,s_{ij}$. Moreover, as a $G$-graded vector spaces, $\mathbb{F}H_im_\ell\mathbb{F}H_j\cong(\mathbb{F}(H_iH_j))^{[\deg m_\ell]}$, for each $\ell=1,2,\ldots,s_{ij}$. Also, $s_{ij}\le|\mathcal{H}_{ij}|$.
\end{Thm}

Finally, given that the grading group is abelian, we solve the problem of classifying the isomorphism classes of group gradings on the incidence algebras (Theorem \ref{Thm_iso}).

\begin{Example}\label{ex2}
In Example \ref{ex1}, the graded algebra $I(X)$ is isomorphic to
\[
\left(\begin{array}{cc}
\mathbb{F}&M\\
0&\mathbb{F}\mathbb{Z}_2
\end{array}\right),
\]
where $M$ is isomorphic to $\mathbb{F}\mathbb{Z}_2$ as a right $\mathbb{F}\mathbb{Z}_2$-module.
\end{Example}

\section{Preliminaries\label{S_Preliminaries}}
\subsection{Incidence algebras} We provide the definition of an incidence algebra over a field $\mathbb{F}$. Let $(X,\le)$ be any partially ordered set (poset, for short). Assume that $(X,\le)$ is \emph{locally finite}, i.e., for all $x,y\in X$, there exists a finite number of $z\in X$ such that $x\le z\le y$. Define $I(X)=\{f:X\times X\to \mathbb{F}\mid f(x,y)=0,\forall x\not\le y\}$. Then $I(X)$ has a natural sum (pointwise sum), and natural scalar multiplication, which gives $I(X)$ a structure of $\mathbb{F}$-vector space. For $f,g\in I(X)$, we define $h=f\cdot g$ as the function $h$, such that $h(x,y)=\sum_{z\in X} f(x,z)g(z,y)$. Note that the unique possibly non-zero elements in the previous sum are the $z\in X$ such that $x\le z\le y$; hence, since $X$ is locally finite, the sum is well-defined. So $h\in I(X)$. It is standard to prove that $I(X)$, with the defined operations, is an associative algebra. The algebra $I(X)$ is called an \emph{incidence algebra}.

Note that the defined multiplication on $I(X)$ is similar to the product of matrices. Moreover, if $X$ is totally ordered and contains $n$ elements, then $I(X)\cong UT_n$. In connection, if $(X,\le_X)$ is arbitrary (not necessarilly totally ordered) and finite with $n$ elements, then we can rename the elements of $X=\{1,2,\ldots,n\}$ in such way that $i\le_X j$ implies $i\le j$ in the usual ordering of the integers. With this identification, we see that $I(X)\subset UT_n$ is a subalgebra. Thus, finite-dimensional incidence algebras are subalgebras of $UT_n$, generated by matrix units, and containing the diagonal matrices. As we are interested in finite-dimensional incidence algebras, we will assume from now on that $I(X)\subset UT_n$.

The incidence algebras are very interesting by its own and, moreover, they are related with other branches of Mathematics. They give rise to interesting and chalenging combinatorial problems. For an extensive theory on incidence algebras see, for instance, the book \cite{SpieDon}.

\subsection{Graded Jacobson radical}
It is fundamental in our arguments to guarantee that the Jacobson radical of the incidence algebras are graded ideals.

For this, we have the following result, due to Gordienko:
\begin{Lemma}[Corollary 3.3 of \cite{Gord}]
Let $A$ be a finite-dimensional algebra, graded by any group $G$. Suppose $\mathrm{char}\,\mathbb{F}=0$ or $\mathrm{char}\,\mathbb{F}>\dim A$. Then the Jacobson radical $J(A)$ is a graded ideal.\qed
\end{Lemma}

Now, if $I(X)$ has a $G$-grading, and $G$ is abelian, then the commutator is a homogeneous operation. Moreover, $[I(X),I(X)]=J(I(X))$ is the Jacobson Radical of the incidence algebra $I(X)$. Hence, $J(I(X))$ is a graded ideal. Thus, we proved

\begin{Lemma}\label{radical_graded}
Let $X$ be finite, and let $I(X)$ be $G$-graded. Assume at least one of the following conditions: $\mathrm{char}\,\mathbb{F}=0$, $\mathrm{char}\,\mathbb{F}>\dim I(X)$, or $G$ is abelian. Then the Jacobson radical of $I(X)$ is graded.\qed
\end{Lemma}

\subsection{Group algebras over fields}
We include in this subsection a few facts concerning the group algebras, which will be essential for our purposes. We do not include the proofs, since the results are classical, or an adaptation of classical proofs when the base field is $\mathbb{C}$, the field of complex numbers.

Let $\mathbb{F}$ be a field, $G$ a group, and we denote $\hat G=\{\chi:G\to\mathbb{F}\text{ group homomorphisms}\}$.

\begin{Lemma}\label{equiv_group_algebra}
Assume that $G$ is abelian finite. The following assertions are equivalent:
\begin{enumerate}
\renewcommand{\labelenumi}{(\roman{enumi})}
\item $G\cong\hat G$,
\item $\mathbb{F}G\cong\mathbb{F}^{\hat G}$ (as $\mathbb{F}$-algebras),
\item for every $g\in G$ having order $q$, $\mathrm{char}\,\mathbb{F}$ does not divide $q$, and $\mathbb{F}$ contains a primitive $q$-root of unit.
\end{enumerate}\qed
\end{Lemma}
For instance, we will have the following situation. A $G$-graded algebra $A$ admits a subalgebra $D\subseteq A$ which is, simultaneously isomorphic to $\mathbb{F}^m$ as ordinary algebras, and to $\mathbb{F}H$ as graded algebras, for some abelian subgroup $H\subseteq G$. Then, in this case, the field $\mathbb{F}$ must satisfy the properties (iii) of the lemma above.

\begin{Lemma}\label{field}
Let $G$ be finite, assume that $\hat G\cong G$, and let $\hat G=\{\chi_1,\ldots,\chi_m\}$. The map $\psi:\mathbb{F}G\to\mathbb{F}^m$, given by
\begin{equation}\label{iso_fields}
\psi(g)=(\chi_1(g),\ldots,\chi_m(g))
\end{equation}
is an isomorphism of algebras. Moreover, every isomorphism of algebras $\psi':\mathbb{F}G\to\mathbb{F}^m$ is given by \eqref{iso_fields}, up to a permutation of the indexes.\qed
\end{Lemma}

\section{Group gradings on incidence algebras\label{S_groupgradings}}
Let $(X,\le_X)$ be a finite poset, and let $n=|X|$. We fix a field $\mathbb{F}$ and $G$ any group with multiplicative notation and neutral element $1$, and we assume one of the conditions of Lemma \ref{radical_graded}. As mentioned before, we can rename the elements of $X$ as $X=\{1,2,\ldots,n\}$, and $i\le_X j$ implies $i\le j$ in the usual ordering of the integers. Hence, we can identify $I(X)\subset UT_n$. If $1\le i\le j\le n$ and $x\in I(X)$, $x(i,j)$ denotes the entry $(i,j)$ of the element $x$.

By direct computation, it is easy to see that for any invertible $M\in I(X)$, we have $M^{-1}\in I(X)$, and it is a well-known result.

\begin{Lemma}\label{ref_here}
	Let $e\in I(X)$ be an idempotent. Then there exists $M\in I(X)$ such that $M^{-1}eM$ is diagonal.
\end{Lemma}
\begin{proof}
	The proof is similar to the proof of Lemma 1 of \cite{VaZa2007}.

	We assume that $I(X)$ acts on the left on the $n\times 1$ matrices. For each $i=1,2,\ldots,n$, let
	$$
		e_i=\left(\begin{array}{c}0\\0\\\vdots\\1\\\vdots\\0\end{array}\right),
	$$
	having non-zero entry only in the position $(i,1)$. One, and only one, between $e\cdot e_i$ or $e_i-e\cdot e_i$ will have non-zero entry at $(i,1)$, where $e\cdot e_i$ is the usual matrix multiplication. Let $f_i$ be either $e\cdot e_i$, or $e_i-e\cdot e_i$, the unique matrix with the non-zero entry in $(i,1)$.

	Let $M$ be the matrix such that the columns are $f_1,f_2,\ldots,f_n$. By construction, $M\in I(X)$. Moreover, since the diagonal entries of $M$ are non-zero, it is invertible. Also, $f_1,f_2,\ldots,f_n$ are eigenvectors of $e$. Thus, $M^{-1}eM$ is diagonal.
\end{proof}

A classical result says that a set of diagonalizable matrices, which pairwise commutes, is simultaneously diagonalizable. We have a similar result for $I(X)$:
\begin{Lemma}\label{simultaneously_diag}
Let $e_1,\ldots,e_s\in I(X)$ be orthogonal idempotents. Then there exists $M\in I(X)$ such that $M^{-1}e_iM$ is diagonal, for each $i=1,2,\ldots,s$.
\end{Lemma}
\begin{proof}
The proof will be by induction on $s$. If $s=1$, then the result follows from the previous lemma. So, assume $s>1$, and $f_1,\ldots,f_n$ is a basis such that $e_1,\ldots,e_{s-1}$ are diagonal, and the matrix $M$ having $f_1,\ldots,f_n$ as its columns lies in $I(X)$.

Fix $j\in\{1,\ldots,n\}$. If $e_sf_j=0$, then $f_j$ is an eigenvector of $e_s$ with eigenvalue $0$, and there is nothing to do. If there is $i<s$ such that $e_if_j=f_j$, then $e_sf_j=e_se_if_j=0$, hence we are in the previous case. So, assume that $e_sf_j\ne0$ and $e_if_j=0$, for all $i<s$. As in the previous lemma, we can change $f_j$ to $e_sf_j$ or $f_j-e_sf_j$, in such a way that this new $f_j$ still form a basis $\{f_1,\ldots,f_n\}$ of $\mathbb{F}^n$. By construction, $e_if_j=0$, for this new $f_j$, and $e_sf_j$ is either $0$ or $f_j$. Continuing the process, we obtain $f_1,\ldots,f_s$ such that the matrix $M$, having these elements as its columns, is in $I(X)$, and every $e_i$ is diagonal in this basis.
\end{proof}

\begin{Lemma}\label{idem_elem}
	Let $x\in UT_n$ be such that $x=e+b$, where $e^2=e\ne0$ is a diagonal matrix, and $b$ is a strict upper triangular matrix. Then there exists a non-zero idempotent in $\mathrm{Span}\{x,x^2,\ldots\}$. In other words, there exists a non-zero polynomial $p(\lambda)\in \mathbb{F}[\lambda]$, such that $p(0)=0$ and $p(x)$ is a non-zero idempotent.
\end{Lemma}
\begin{proof}
Let $\bar{\mathbb{F}}$ an algebraic closure of $\mathbb{F}$. By Jordan canonical form, there exists an invertible matrix $m$ such that
$$
	mxm^{-1}=\left(\begin{array}{cccccc}1&&\ast&0&\cdots&0\\
		&\ddots&&\vdots&&\vdots\\
						&&1&0&\cdots&0\\
						&&&0&&\ast\\
						&&&&\ddots&\\
						&&&&&0
	\end{array}\right).
$$
With an adequate $s>1$, we obtain
$$
	mx^sm^{-1}=\left(\begin{array}{cc}J&0\\0&0\end{array}\right),
$$
	where $J$ is a triangular matrix, with $1$'s in the diagonal. Now, there exists a polynomial $p(\lambda)$ such that $p(J)=1$, the identity matrix. Indeed, let $q(\lambda)=\sum_{i\ge0} a_i\lambda^i$ be the characteristic polynomial of $J$. By the Cayley-Hamilton Theorem, $q(J)=0$. Moreover, since $q(0)=a_0=\det J=1$, we obtain that $p(\lambda)=\sum_{i>0}-a_i\lambda^i$ satisfies $p(J)=1$.

	Let $e'=p(x^s)$, then $me'm^{-1}=\left(\begin{array}{cc}1&0\\0&0\end{array}\right)$, hence $e^{\prime2}=e'$, as desired.
\end{proof}

Fix any $G$-grading on $I(X)$. Define
$$
\mathscr{E}(X)=\{\text{$e\in I(X)$ non-zero homogeneous idempotent}\}.
$$
Of course it is non-empty since $1\in\mathscr{E}(X)$. For $e_1,e_2\in\mathscr{E}(X)$, let $e_1\preceq e_2$ denote (the non-necessarily partial order) $e_1I(X)\subseteq e_2I(X)$ (usually, $\preceq$ is called a preordering). Let $e\in\mathscr{E}(X)$ be a minimal element according to $\preceq$ (since $I(X)$ has finite dimension, we can always find such minimal $e$). Note that $\deg e=1$, since $e^2=e$. By Lemma \ref{ref_here}, we can assume $e$ diagonal, up to a graded isomorphism. The algebra $B=eI(X)e$ has a natural $G$-grading, induced from the grading of $I(X)$. We will prove that $B$ is a (commutative) graded division algebra.

\begin{Lemma}\label{finite_order}
	For every homogeneous $x\in B$, $x\notin J(B)$, $\deg x$ has finite order.
\end{Lemma}
\begin{proof}
	Note that, if $x\notin J(B)$, then $x$ has at least one non-zero entry in the diagonal. So, $x^m\notin J(B)$ for all $m\in\mathbb{N}$. In particular, $x,x^2,x^3,\ldots$ are non-zero elements. If $\deg x$ has infinite order, then the elements $x,x^2,\ldots$ have all different degree, thus they are linearly independent. However, $B$ has finite dimension, so we obtain a contradiction. As a conclusion, $\deg x$ has finite order.
\end{proof}

Recall, from Lemma \ref{radical_graded}, that $J(B)$ is graded. Thus $B/J(B)$ has a natural $G$-grading, induced from the grading on $B$.
\begin{Lemma}\label{mainlemma}
	\begin{enumerate}
		\renewcommand{\labelenumi}{(\roman{enumi})}
		\item The algebra $B/J(B)$ is a commutative graded division algebra.
		\item Write $B/J(B)=\bigoplus_{g\in G}\bar B_g$. Then $\dim\bar B_g\le1$, for all $g\in G$.
	\end{enumerate}
\end{Lemma}
\begin{proof}
	If $\dim\bar B_1$ is not $1$, then we claim that there would exists an idempotent $e'\in\bar B_1$ such that $e'\preceq e$. Indeed, let $\bar x\in\bar B_1$ be a non-zero homogeneous element, not a scalar multiple of $e$, modulo $J(B)$.
	Then there exists a homogeneous $x\in B$ of degree $1$, having non-zero entries in the diagonal. We can write $x=\alpha_1e_1+\alpha_2e_2+\cdots+\alpha_se_s+b$, where $b\in J(B)$, $e_1,e_2,\ldots,e_s$ are diagonals orthogonal idempotents (sum of diagonal matrix units), and $\alpha_1,\ldots,\alpha_s\in \mathbb{F}$ are non-zero elements such that $\alpha_i\ne\alpha_j$ for all $i\ne j$. Since $x$ has non-zero entry in the diagonal, $s>0$; and since $x$ and $e$ are linearly independent modulo $J(B)$, we have necessarilly $s>1$. The powers $x,x^2,\ldots,x^s$ are homogeneous elements of degree 1. By Vandermonde argument, we conclude that $y=e_1+b'$ is a linear combination of $x,x^2,\ldots,x^s$, for some $b'\in J(B)$, hence $y$ is homogeneous of degree 1. Using Lemma \ref{idem_elem}, we obtain a homogeneous idempotent $e'$, which is a linear combination of $y,y^2,\ldots$. Thus $e'\in B$. Moreover, we have $e'I(X)\subsetneq eI(X)$. But $e$ is minimal, so we obtain a contradiction. As a conclusion, $\dim\bar B_1=1$.

	Now, let $\bar x\in B/J(B)$ be non-zero and homogeneous. Thus, $\bar x$ has non-zero entry in the diagonal. Hence, $\bar x^m\ne0$, for all $m\in\mathbb{N}$. By Lemma \ref{finite_order}, $\bar x^m\in\bar B_1$, for some $m\in\mathbb{N}$. This means $\bar x^m=\lambda\bar e$, for some $\lambda\ne0$. As a consequence, $\bar x$ is invertible. Thus, $B/J(B)$ is a graded division algebra. Moreover, since $B/J(B)$ is a sum of copies of $\mathbb{F}$, as an ordinary algebra, then it is necessarily commutative, proving (i). Since $\dim\bar B_1=1$ and $B/J(B)$ is a graded division, then $\dim\bar B_g\le1$, for all $g\in G$, thus proving (ii).
\end{proof}

\begin{Cor}\label{gradedalgebra}
	The algebra $B$ is semisimple as an ordinary algebra, that is $J(B)=0$. In particular, $B$ is a commutative graded division algebra with unit $e$.
\end{Cor}
\begin{proof}
	Let $J=J(B)$ and $\bar B=B/J$. Note that $J/J^2$ is a graded $\bar B$-module. The right annihilator $\mathrm{Ann}^r_{\bar B}(J/J^2)$ is a homogeneous subspace of $\bar B$. However, no nonzero homogeneous element of $\bar B$ annihilates a nonzero homogeneous element of $J/J^2$, since $\bar B$ is a graded division algebra, by previous lemma. Hence, $J(B)=0$. Since, by previous lemma, $B/J(B)$ is a commutative graded division algebra, then $B$ is a commutative graded division algebra as well.
\end{proof}

From now on, $e$ will always denote a minimal idempotent as before. Denote by $\mathscr{D}(e)=\{i\in X\mid e(i,i)\ne0\}$.
\begin{Cor}\label{noncomparable}
	The elements of $\mathscr{D}(e)$ are not pairwise comparable.
\end{Cor}
\begin{proof}
	If not, suppose $i<j$, with $i,j\in\mathscr{D}(e)$. Recall that $B=eI(X)e\cong I(Y)$, for some poset $Y$ (where $Y=\mathscr{D}(e)$). Then $e_{ij}\in B$. Moreover, one has $e_{ij}\in J(B)$, so $J(B)\ne0$, contrary to Corollary \ref{gradedalgebra}.
\end{proof}

Now, if $e\ne1$, consider the set
$$
\mathscr{E}_1(X)=\{\text{$e'\in I(X)$ idempotent $\mid$ $e'I(X)\subset(1-e)I(X)$}\}.
$$
We remark that $\mathscr{E}_1(X)\ne\emptyset$, since $1-e\in\mathscr{E}_1(X)$. We claim that we can choose the idempotents belonging to $\mathscr{E}_1(X)$ so that they commute with $e$. Indeed, if $e'\in\mathscr{E}_1(X)$, then $ee'=ee^{\prime2}\in e(1-e)I(X)=0$. Also, let $f_1=e'-e'e$. Then $f_1^2=f_1$, $f_1e=ef_1=0$ and $f_1I(X)=e'I(X)$. Thus, $f_1\in\mathscr{E}_1(X)$ is the required element. Let $e_1=e$, and chose a homogeneous idempotent $e_2\in\mathscr{E}_1(X)$ which is minimal with respect to $\preceq$, and orthogonal to $e_1$.

Proceeding by induction, we assume that we have a set $e_1,\ldots,e_r$ of minimal homogeneous orthogonal idempontents. Then we define
$$
\mathscr{E}_r(X)=\{e'\in I(X)\text{ idempotent }\mid e'I(X)\subseteq(1-e_1\cdots-e_r)I(X)\}.
$$
As before, we can find a minimal homogeneous idempotent $e_{r+1}\in\mathscr{E}_r(X)$ which is orthogonal to $e_1,\ldots,e_r$. We can continue the process, and, since $\dim I(X)$ is finite, it ends within finitely many steps. Let $e_1,e_2,\ldots,e_t$ be the homogeneous idempotents, which are minimal with respect to $\preceq$, obtained by the previous construction. By Lemma \ref{simultaneously_diag}, we can assume $e_1,e_2,\ldots,e_t$ diagonal.
\begin{Def}\label{defassociatedposet}
The diagonal homogeneous idempotents $e_1,\ldots,e_t$ will be called the \emph{minimal idempotents}. We denote $\mathcal{E}=\{e_1,\ldots,e_t\}$.
\end{Def}

We denote $D_i=e_iI(X)e_i$ which is, according to Corollary \ref{gradedalgebra}, a commutative graded division algebra. Moreover, $M_{ij}=e_iI(X)e_j$ is a graded $(D_i,D_j)$-bimodule. The triple $(D_i,D_j,M_{ij})$ constitutes a triangular algebra.

If $e_i$ and $e_j$ are two minimal idempotents, we denote $e_i\trianglelefteq e_j$ if there exist $x\in\mathscr{D}(e_i)$ and $y\in\mathscr{D}(e_j)$ such that $x\le y$.
\begin{Lemma}\label{ordering}
The relation $\trianglelefteq$ is a partial order in $\mathcal{E}$.
\end{Lemma}

Before we prove Lemma \ref{ordering}, we need to investigate the relation between each $M_{ij}$ and the poset.

Recall that the structure of graded modules over graded division algebras is similar to the structure of vector spaces over division rings (see, for instance, \cite[page 29]{EldKoc}). More precisely, let $D$ be a $G$-graded division algebra and let $M$ be a graded right $D$-module. Then $M\cong\bigoplus_{\alpha}M_\alpha$, where each $M_\alpha=\,^{[h_\alpha]}\!D$, for some $h_\alpha\in G$. Here, if $D=\bigoplus_{g\in G}D_g$, then
$$
	^{[h]}\!D=\bigoplus_{g\in G}\,^{[h]}\!D_g,\quad\text{ where $^{[h]}\!D_g=D_{h^{-1}g}$.}
$$
In particular, every graded right $D$-module is free. A similar discussion is valid for graded left $D$-modules.

Let $e_i$ and $e_j$ be minimal idempotents. For each $x\in\mathscr{D}(e_i)$, we define the \emph{link} of $x$ to $e_j$ as the non-negative integer $\ell(x,e_j)=|\{y\in\mathscr{D}(e_j)\mid x\le y\}|$. Similarly, we define the link of $e_i$ to $y$, denoted $\ell(e_i,y)$. The link of $e_i$ to $e_j$ is $\ell(e_i,e_j)=|\{(x,y)\mid x\le y,x\in\mathscr{D}(e_i),y\in\mathscr{D}(e_j)\}|$.
\begin{Lemma}\label{link}
Let $e_i,e_j$ be minimal idempotents. Then, for any $x\in\mathscr{D}(e_i)$, $y\in\mathscr{D}(e_j)$,
\[
\ell(e_i,e_j)=|\mathscr{D}(e_i)|\ell(x,e_j)=\ell(e_i,y)|\mathscr{D}(e_j)|.
\]
\end{Lemma}
\begin{proof}
Let $M_{ij}=e_iI(X)e_j$, which is a free $G$-graded left $D_i$-module.
We shall prove that $\ell(x,e_j)=\dim_{D_i}M_{ij}$, for any $x\in\mathscr{D}(e_i)$. 

On one hand, as $\mathbb{F}$-linear space, we have \[M_{ij}=\bigoplus_{(x,y)\in\mathscr{D}(e_i)\times\mathscr{D}(e_j);x\leq y}\mathbb{F}e_{xy}.\] Thus, $\ell(e_i,e_j)=\dim_\mathbb{F} M_{ij}=\dim_\mathbb{F} D_i\dim_{D_i}M_{ij}$. On the other hand, let $v_1,\ldots,v_\ell$ be a $D_i$-basis of $M_{ij}$, where $\ell>0$ (the case $\ell=0$ is trivial).

\noindent\textbf{Claim.} For $k\in\{1,\ldots,\ell\}$, let $w\in D_iv_k$, and asume $w(x,y)=0$, but $v_k(x,y)\ne0$. Then $w(x,y')=0$, for all $y'\in\mathscr{D}(e_j)$.

Indeed, in this case, $B=\{e_{xx}v_k\mid x\in\mathscr{D}(e_i)\}$ is a set of $|\mathscr{D}(e_i)|$ $\mathbb{F}$-linearly independent elements. Since $\dim_\mathbb{F} D_iv_k=\dim_\mathbb{F}D_i=|\mathscr{D}(e_i)|$, $B$ is a (non-homogeneous) $\mathbb{F}$-vector space basis of $D_iv_k$. So, write $w=\sum_{x\in\mathscr{D}(e_i)}\lambda_xe_{xx}v_k$. If $v_k(x,y)\ne0$, but $w(x,y)=0$, then necessarily $\lambda_x=0$. This implies $w(x,y')=0$, for all $y'\in\mathscr{D}(e_j)$.

As a conclusion, since $v_1,\ldots,v_\ell$ generate $M_{ij}$ as a $D_i$-linear space, every element of $M_{ij}$ is a linear combination of elements in $D_iv_1,\ldots,D_iv_\ell$. By the claim, in order to be possible to fill all the entries of the coordinates of type $(x,\cdot)$ in $M_{ij}$, it is necessary that $\ell\ge\ell(x,e_j)$. Hence, $\ell\ge\max\{\ell(x,e_j)\mid x\in\mathscr{D}(e_i)\}$. And, therefore,
\[\dim_\mathbb{F} D_i\dim_{D_i}M_{ij}=\ell(e_i,e_j)=\sum_{x\in\mathscr{D}(e_i)}\ell(x,e_j)\leq|\mathscr{D}(e_i)|\ell=\dim_\mathbb{F} D_i\dim_{D_i}M_{ij}\] which implies $\ell(x,e_j)=\dim_{D_i}M_{ij}$, for any $x\in\mathscr{D}(e_i)$ and we are done.
\end{proof}

Now, we can prove the lemma:
\begin{proof}[Proof of Lemma \ref{ordering}]
It is clear that $\trianglelefteq$ is reflexive and transitive. So, let $e_i\ne e_j$ (thus, by construction, $\mathscr{D}(e_i)\cap\mathscr{D}(e_j)=\emptyset$), and assume that $e_i\trianglelefteq e_j$, and $e_j\trianglelefteq e_i$. Then, by Lemma \ref{link}, every $x\in\mathscr{D}(e_i)$ is $\le$ to at least one $y\in\mathscr{D}(e_j)$. In the same way, every $y\in\mathscr{D}(e_j)$ is $\le$ to at least one $x\in\mathscr{D}(e_i)$. Hence, we can find $x_1,x_2\in\mathscr{D}(e_i)$, and $y\in\mathscr{D}(e_j)$ such that $x_1<y<x_2$, so $x_1<x_2$. This contradicts Corollary \ref{noncomparable}. Hence, $\trianglelefteq$ is antisymmetric.
\end{proof}

As a consequence of Lemma \ref{ordering}, we can rename the elements of $X$, in such a way that $\mathscr{D}(e_1)=\{1,2,\ldots,i_1\}$, $\mathscr{D}(e_2)=\{i_1+1,i_1+2,\ldots,i_2\}$, ..., $\mathscr{D}(e_t)=\{i_{t-1}+1,\ldots,n\}$. Since $e_iI(X)e_j=\mathrm{Span}\{e_{k\ell}\mid k\in\mathscr{D}(e_i),\ell\in\mathscr{D}(e_j),k<\ell\}$, we have $\dim e_iI(X)e_j\le\dim D_i\dim D_j$. Moreover, as graded vector space, $I(X)=\sum_{i,j}e_iI(X)e_j$. In this way, we obtain the proof of the Theorem \ref{class}.
\begin{Cor}\label{firstpartthm}
Up to a graded isomorphism, there exists commutative graded division algebras $D_1,D_2,\ldots,D_t$, all of them having all homogeneous subspaces with dimension at most 1, such that
$$
I(X)\cong\left(\begin{array}{cccc}D_1&M_{1,2}&\ldots&M_{1,t}\\&D_2&\ddots&\vdots\\&&\ddots&M_{t-1,t}\\&&&D_t\end{array}\right),
$$
where each $M_{i,j}$ is a graded $(D_i,D_j)$-bimodule of dimension at most $\dim_\mathbb{F} D_i\dim_\mathbb{F}D_j$.\qed
\end{Cor}

If $D$ is a commutative graded division algebra, where each homogeneous component has dimension at most 1 and $D\cong\mathbb{F}^{\dim D}$, then $D\cong \mathbb{F}H$, where $H=\mathrm{Supp}\,D$. Indeed, a set of nonzero homogeneous elements $\{X_u\mid u\in H\}$, where $\deg X_u=u$ is a vector space basis of $D$. Also, $X_uX_v=\sigma(u,v)X_{uv}$, for some (nonzero) $\sigma(u,v)\in\mathbb{F}$. Since $D$ is associative, we obtain that $\sigma$ is a 2-cocycle. Thus, $\sigma\in Z^2(H,\mathbb{F}^\times)$, and $D\cong\mathbb{F}^\sigma H$ (the twisted group algebra). To complete the proof, we need the following result, communicated by M. Kochetov:
\begin{Prop}\label{kochetov}
Let $H$ be a finite abelian group and $\sigma\in Z^2(H,\mathbb{F}^\times)$. If $\mathbb{F}^\sigma H\cong\mathbb{F}^{|H|}$, then $\mathbb{F}H\cong\mathbb{F}^{|H|}$.
\end{Prop}
\begin{proof}
Given a finite-dimensional unital commutative algebra $\mathcal{A}$ over $\mathbb{F}$, let $N(\mathcal{A})$ denote the number of algebra homomorphisms from $\mathcal{A}$ to $\mathbb{F}$.

\noindent\textbf{Claim.} $\mathcal{A}$ is a direct sum of copies of $\mathbb{F}$ if and only if $N(\mathcal{A})=\dim\mathcal{A}$.

Indeed, $(\Rightarrow)$ is obvious, and $(\Leftarrow)$ follows from linear independence of distinct homomorphisms.

Note that both $N$ and $\dim$ are multiplicative over tensor products of algebras.

Now, write $H=\langle h_1\rangle\times\cdots\times\langle h_r\rangle$, direct product of cyclic groups, where each $h_i$ has order $m_i$. Then, we notice that
$$
\mathbb{F}^\sigma H\cong\mathbb{F}[x_1]/(x_1^{m_1}-a_1)\otimes\cdots\otimes\mathbb{F}[x_r]/(x_r^{m_r}-a_r),
$$
where $a_i=\prod_{j=1}^{m_i-1}\sigma(h_i,h_i^j)\in\mathbb{F}$. In the same way, $\mathbb{F}H\cong\bigotimes_{i=1}^r\mathbb{F}[x_i]/(x_i^{m_i}-1)$.

From the remarks above, $\mathbb{F}^\sigma H$ is a direct sum of copies of $\mathbb{F}$ if and only if each $\mathbb{F}[x_i]/(x_i^{m_i}-a_i)$ is a direct sum of copies of $\mathbb{F}$. The former condition is equivalent to $\mathbb{F}$ contains $m_i$ distinct $m_i$-roots of $a_i$; and this implies that $\mathbb{F}$ contains a primitive $m_i$-th root of $1$. Thus, each $\mathbb{F}[x_i]/(x_i^{m_i}-1)$ is a direct sum of $\mathbb{F}$'s. Hence, by the remarks above, $\mathbb{F}H$ is a direct sum of copies of $\mathbb{F}$.
\end{proof}
As a consequence, $\mathbb{F}^\sigma H\cong\mathbb{F}H$ as ungraded algebras. From \cite[Corollary 1.2]{karpilovsky}, we obtain that $\alpha$ is a coboundary map. Thus, $\mathbb{F}^\sigma H\cong\mathbb{F}H$ as graded algebras.

\begin{Remark}\label{remarkassociatedposet}
Consider the notation of Corollary \ref{firstpartthm}. Let $e_i\in D_i$ be the unity. Set $\mathcal{E}=\{1,2,\ldots,t\}$. We define the partial order $\trianglelefteq$ on $\mathcal{E}$ by the following (equivalent) conditions (see Lemma \ref{ordering}):
\begin{enumerate}
\item $i\trianglelefteq j$,
\item $e_i\trianglelefteq e_j$,
\item there exists $x\in\mathscr{D}(e_i)$ and $y\in\mathscr{D}(e_j)$ such that $i\le j$,
\item $M_{ij}\ne0$.
\end{enumerate}
Note that we can identify the present definition of $\mathcal{E}$ with that given in Definition \ref{defassociatedposet}. We call $\mathcal{E}$ the \emph{associated poset} of the graded algebra $I(X)$.
\end{Remark}

Note that we have the following equality:
\[
|H_i|=\dim D_i=|\mathscr{D}(e_i)|.
\]




\section{On the structure of graded bi-vector spaces\label{S_bivectorspace}}

In this section, we shall investigate graded bimodules. We are interested in the special case where $G$ is an abelian group, $H_1,H_2\subseteq G$ are finite subgroups, and $V$ is a $G$-graded $(\mathbb{F}H_1,\mathbb{F}H_2)$-bimodule. We will conclude that, in this situation, $V$ is a $\mathbb{F}H$-vector space, where $H=H_1H_2$. Hence, the bimodules have a nice description in this case.

The notion of tensor product will play a fundamental role in this section. Recall that, if $R$ is any $\mathbb{F}$-algebra, $M,N$ are $\mathbb{F}$-vector spaces, $M$ is a right $R$-module, and $N$ is a left $R$-module, then $M\otimes_RN$ is a uniquely defined $\mathbb{F}$-vector space, satisfying a universal property. It is spanned by $m\otimes n$, where $m\in M$, and $n\in N$, and they satisfy $mr\otimes n=m\otimes rn$, $r\in R$. For the special case where $M$ and $N$ are algebras, then $M\otimes_RN$ is an algebra as well.

We are interested in the following special case: let $R,A_1,A_2$ be $\mathbb{F}$-algebras, and assume $\varphi_i:R\to A_i$ homomorphism of algebras, for $i=1,2$. Then $A_1$ becomes a right $R$-module by $a\cdot r=a\varphi_1(r)$, and, similarly, $A_2$ becomes a left $R$-module. So we can construct the tensor product $A_1\otimes_RA_2$. Since $\varphi_1$ and $\varphi_2$ will play a somewhat significant role, we use the notation $(A_1\otimes_RA_2,\varphi_1,\varphi_2)$. In this case, $A_1\otimes_RA_2$ is spanned by $a_1\otimes a_2$, for $a_i\in\beta_i$, where $\beta_i$ is any basis of $A_i$; also, we have $b_1\varphi_1(r)\otimes b_2=b_1\otimes\varphi_2(r)b_2$, for any $r\in R$, $b_1\in A_1$, $b_2\in A_2$.

Now, assume further that $R,A_1,A_2$ are $G$-graded algebras, and $\varphi_1,\varphi_2$ are $G$-graded homomorphism of algebras. So $A_1$ and $A_2$ are $G$-graded $R$-modules. The algebra $A_1\otimes_RA_2$ admits a natural structure of $G$-grading, where the homogeneous component of degree $g\in G$ is spanned by $a_1\otimes a_2$, $a_i\in A_i$ homogeneous satisfying $\deg a_1\deg a_2=g$. Moreover, $A_1\otimes_RA_2$ becomes a $G$-graded algebra. In what follows, we will see that, in the case of group algebras, the just defined grading has a nicer way to be constructed.

\begin{Lemma}\label{Lem0}
Let $G$ be an abelian group, and $H_1,H_2\subseteq G$ be finite subgroups. Let $H=H_1H_2$, and $\mathcal{H}=H_1\cap H_2$. Let $\iota_j:\mathbb{F}\mathcal{H}\to\mathbb{F}H_j$ be $G$-graded monomorphisms, for $j=1,2$. Then, as graded algebras,
$$
(\mathbb{F}H_1\otimes_{\mathbb{F}\mathcal{H}}\mathbb{F}H_2,\iota_1,\iota_2)\cong\mathbb{F}H.
$$
\end{Lemma}
\begin{proof}
First we note that $\mathbb{F}H_1\otimes_{\mathbb{F}\mathcal{H}}\mathbb{F}H_2$ has a natural well-defined $G$-grading, given by $\deg h_1\otimes h_2=h_1h_2$, where $h_1\in H_1$, $h_2\in H_2$.

Now, since every homogeneous component of $\mathbb{F}H_j$ has dimension at most 1, we see that an inclusion $\iota_j(h)=\chi_j(h)h$, for each $h\in\mathcal{H}$, for some $\chi_j(h)\in\mathbb{F}$. Moreover, $h\mapsto\chi_j(h)\in\mathbb{F}$ is a character of $H$. Consider the character given by $\chi_j^{-1}$. It is known that every character of $\mathcal{H}$ can be extended to a character of $H_j$. 
 So, let $\bar\chi_j$ be an extension of $\chi_j^{-1}$.

Let $\varphi:\mathbb{F}H_1\otimes_{\mathbb{F}\mathcal{H}}\mathbb{F}H_2\to\mathbb{F}H$ be defined by $\varphi(h_1\otimes h_2)=\bar\chi_1(h_1)\bar\chi_2(h_2)h_1h_2$. If $h\in\mathcal{H}$, then
\begin{align*}
&\varphi(h_1\iota_1(h)\otimes h_2)=\chi_1(h)\bar\chi_1(h_1h)\bar\chi_2(h_2)h_1hh_2=\chi_1(h)\bar\chi_1(h)\bar\chi_1(h_1)\bar\chi_2(h_2)h_1hh_2,\\
&\varphi(h_1\otimes\iota_2(h)h_2)=\chi_2(h)\bar\chi_1(h_1)\bar\chi_2(hh_2)h_1hh_2=\chi_2(h)\bar\chi_2(h)\bar\chi_1(h_1)\bar\chi_2(h_2)h_1hh_2.
\end{align*}
Since $\chi_1(h)\bar\chi_1(h)=\chi_2(h)\bar\chi_2(h)=1$, we obtain that $\varphi$ is well-defined.

Since $G$ is abelian, $\varphi$ is an onto homomorphism of graded algebras.

Let $\sum_{i=1}^m\lambda_ig_i\otimes h_i\in\mathrm{Ker}\,\varphi$. As $\mathrm{Ker}\,\varphi$ is a graded subspace, we can suppose that $g_1h_1=\ldots=g_mh_m$. Moreover, we can assume that $h_1,\ldots,h_m$ are $\mathbb{F}\mathcal{H}$-linearly independent. Thus, $g_ih_i=g_jh_j$ implies $h_ih_j^{-1}\in\mathcal{H}$, which is a contradiction if $m>1$. Hence, $m\le 1$, and this is sufficient to conclude that $\mathrm{Ker}\,\varphi=0$.
\end{proof}


\begin{Lemma}\label{Lem1}
Let $G$ be an abelian group, $H_1,H_2\subseteq G$ finite subgroups, $\mathcal{H}=H_1\cap H_2$. Let $V$ be a finite-dimensional $G$-graded $(\mathbb{F}H_1,\mathbb{F}H_2)$-bimodule. Then, there exist $m_1,\ldots,m_s\in V$ homogeneous, and $\chi_1,\ldots,\chi_s\in\hat{\mathcal{H}}$, such that
$$
V=\mathbb{F}H_1m_1\mathbb{F}H_2\oplus\cdots\oplus\mathbb{F}H_1m_s\mathbb{F}H_2,
$$
and $hm_i=\chi_i(h)m_ih$, for all $i=1,\ldots,s$, $h\in\mathcal{H}$.
\end{Lemma}
\begin{proof}
Let $W_g$ be the homogeneous component of degree $g$ of $V$. Given $h\in\mathcal{H}$, let $L_h,R_h:V\to V$ denote the maps $L_h(m)=hm$, $R_h(m)=mh$. Then $L_h$ is a linear isomorphism from $W_g$ to $W_{hg}$. Also, $R_{h^{-1}}:W_{gh}\to W_g$ is an isomorphism as well. Thus, $L_hR_{h^{-1}}\in\mathrm{GL}(W_g)$. It is elementary to see that this defines a linear representation $h\in\mathcal{H}\to L_hR_{h^{-1}}\in\mathrm{GL}(W_g)$.

Since $\mathcal{H}$ is abelian, the representation is completely reducible, and each irreducible representation has degree 1. So, we can find a vector space basis $v_1,\ldots,v_r\in W_g$, and characters $\chi_1,\ldots,\chi_r\in\mathcal{H}$ such that $L_hR_{h^{-1}}(v_i)=\chi_i(h)v_i$, for all $i=1,\ldots,r$. In particular, this gives $hv_i=\chi_i(h)v_ih$, for all $h\in\mathcal{H}$.

It is clear that every component $W_{g'}$, where $g'\in H_1gH_2$ belongs to $\mathbb{F}H_1v_1\mathbb{F}H_2\oplus\cdots\oplus\mathbb{F}H_1v_r\mathbb{F}H_2$. Thus, we can take $g_2\notin H_1gH_2$ and repeat the process. Since $\dim V<\infty$, the process ends.
\end{proof}

In the notation of the previous lemma, given $g_1\in H_1$ and $g_2\in H_2$, we have
$$
g_1hm_ig_2=\chi_i(h)g_1m_ihg_2,\forall h\in\mathcal{H}.
$$

\begin{Lemma}\label{Lem2}
Assuming the same hypotheses of the previous lemma, let $\chi\in\hat{\mathcal{H}}$, and $V=\mathbb{F}H_1m\mathbb{F}H_2$, where $hm=\chi(h)mh$, for all $h\in\mathcal{H}$. Then $V$ is a graded $\mathbb{F}H_1\otimes_{\mathbb{F}\mathcal{H}}\mathbb{F}H_2$-module.
\end{Lemma}
\begin{proof}
Consider the inclusions
\begin{align*}
&\iota_1:h\in\mathbb{F}\mathcal{H}\to h\in\mathbb{F}H_1,\\
&\iota_2:h\in\mathbb{F}\mathcal{H}\to\chi(h)h\in\mathbb{F}H_2.
\end{align*}
Define the left action of $\mathbb{F}H_1\otimes_{\mathbb{F}\mathcal{H}}\mathbb{F}H_2$ on $V$ by $g_1\otimes g_2m=g_1mg_2$. It suffices to prove that this action is well-defined. For, assume $g_1\in H_1$, $g_2\in H_2$, $h\in\mathcal{H}$. Then $g_1h\otimes g_2=\chi(h)g_1\otimes hg_2$, and
$$
(g_1h\otimes g_2)m=g_1hmg_2=\chi(h)g_1mhg_2=(g_1\otimes\chi(h)hg_2)m.
$$
\end{proof}

Recall that, if $R$ and $S$ are any rings, $\varphi:R\to S$ is a ring homomorphism, and $M$ is an $S$-module; then $M$ becomes an $R$-module if we set $m\cdot r=m\varphi(r)$, for $m\in M$ and $r\in R$. We summarize the results of this section. 
\begin{Cor}\label{main_cor}
Let $G$ be an abelian group, $H_1,H_2\subseteq G$ finite subgroups, $\mathcal{H}=H_1\cap H_2$, $H=H_1H_2$. Let $V$ be a finite-dimensional $G$-graded $(\mathbb{F}H_1,\mathbb{F}H_2)$-bimodule. Then, there exist homogeneous $m_1,\ldots,m_s\in V$ and characters $\chi_1,\ldots,\chi_s\in\hat{\mathcal{H}}$ such that
\begin{equation}\label{decomp}
V=\mathbb{F}H_1m_1\mathbb{F}H_2\oplus\cdots\oplus\mathbb{F}H_1m_s\mathbb{F}H_2,
\end{equation}
where $hm_\ell=\chi_\ell(h)m_\ell h$, for each $h\in\mathcal{H}$. Moreover, for each $i$, one has $\mathbb{F}H_1m_i\mathbb{F}H_2\cong(\mathbb{F}H)^{[\deg m_i]}$, as graded vector spaces.
\end{Cor}
\begin{proof}
Lemma \ref{Lem1} says that we can find homogeneous $m_1,\ldots,m_s\in V$ such that
$$
V=\mathbb{F}H_1m_1\mathbb{F}H_2\oplus\cdots\oplus\mathbb{F}H_1m_s\mathbb{F}H_2.
$$
By Lemma \ref{Lem2}, every $\mathbb{F}H_1m_i\mathbb{F}H_2$ has a structure of graded $(\mathbb{F}H_1\otimes_{\mathbb{F}\mathcal{H}}\mathbb{F}H_2,\iota_{i1},\iota_{i2})$-module. By Lemma \ref{Lem0}, there exists a graded isomorphism $\varphi_i:\mathbb{F}H\to(\mathbb{F}H_1\otimes_{\mathbb{F}\mathcal{H}}\mathbb{F}H_2,\iota_{i1},\iota_{i2})$. Hence, $\mathbb{F}H_1m_i\mathbb{F}H_2$ has a structure of graded $\mathbb{F}H$-module, and it is $1$-dimensional, since it is $\mathbb{F}H$-spanned by $m_i$. This concludes the proof.
\end{proof}


\begin{Remark}
It should be noted that Corollary \ref{main_cor} is no longer valid when the grading group is not abelian. Indeed, let $H$ be a finite abelian group, $C_2=\{1,w\}$ the cyclic group of order 2, and let $G=H\ast C_2$ be the free product of $H$ and $C_2$. Consider $V=\mathbb{F}H\otimes_\mathbb{F}\mathbb{F}H$, and define the $G$-grading on $V$ by $V=\bigoplus_{g\in G}V_g$, where
$$
V_{h_1wh_2}=\mathrm{Span}\{h_1\otimes h_2\},\quad h_1,h_2\in H.
$$
Then $V$ is a $G$-graded vector space, and $V$ has a structure of $G$-graded $(\mathbb{F}H,\mathbb{F}H)$-bimodule. For every nonzero homogeneous $v\in V$, we have $\mathbb{F}Hv\mathbb{F}H=V$; and it is not isomorphic to a shift of $\mathbb{F}H$. Moreover, $V$ is not a $G$-graded left $\mathbb{F}H\otimes_F\mathbb{F}H$-module.
\end{Remark}

Now, we are able to prove Theorem \ref{Thm3}.

\begin{proof}[Proof of Theorem \ref{Thm3}]
Fix $i,j$. We consider the triangular algebra
$$
A=\left(\begin{array}{cc}\mathbb{F}^{m_i}&V_{ij}\\&\mathbb{F}^{m_j}\end{array}\right),
$$
where $m_\ell=|H_\ell|$, which is graded isomorphic to $\left(\begin{array}{cc}\mathbb{F}H_i&M_{ij}\\&\mathbb{F}H_j\end{array}\right)$. Let $e_\ell\in H_\ell$ be the unit. Then $A=(e_i+e_j)I(X)(e_i+e_j)$, and it is the incidence algebra of the subposet $Y=\mathscr{D}(e_i)\cup\mathscr{D}(e_j)\subseteq X$.

Let $\mathcal{H}=H_i\cap H_j$, $\hat{\mathcal{H}}=\{\chi_1,\ldots,\chi_m\}$, and assume $\hat{\mathcal{H}}=\hat{H}_i\cap\hat{H}_j$. Write
$$
\hat{H}_\ell=\chi_1^{(\ell)}\hat{\mathcal{H}}\dot\cup\ldots\dot\cup\chi_{n_\ell}^{(\ell)}\hat{\mathcal{H}},
$$
where $m_\ell=mn_\ell$, for $\ell\in\{i,j\}$ (we can assume $\chi_1^{(\ell)}=1$).

By Lemma \ref{field}, up to renaming the entries, the graded isomorphism $\phi:\mathbb{F}H_i\oplus M_{ij}\oplus\mathbb{F}H_j\to A$ is such that, for $h_1\in H_i$, $h_2\in H_j$,
\begin{equation}\label{iso_char}
\phi(h_1+h_2)=\left(\begin{array}{cc}
(\chi_1^{(i)}\chi_1(h_1),\ldots,\chi_{n_i}^{(i)}\chi_m(h_1))&0\\
&(\chi_1^{(j)}\chi_1(h_2),\ldots,\chi_{n_j}^{(j)}\chi_m(h_2))\end{array}\right).
\end{equation}
Let $M_{ij}=\mathbb{F}H_im_1\mathbb{F}H_j\oplus\cdots\oplus\mathbb{F}H_im_s\mathbb{F}H_j$, with characters $\mu_1,\ldots,\mu_s$, as in \eqref{decomp}. We only need to prove that $\mu_1,\ldots,\mu_{s}$ are pairwise distinct characters. By Corollary \ref{main_cor}, $\dim\mathbb{F}H_im_\ell\mathbb{F}H_j=|H_iH_j|$, so $\phi(m_\ell)\in V_{ij}$ has at least $|H_iH_j|$ non-zero entries.

To make notation easier, we will identify the elements of $\mathscr{D}(e_i)\cup\mathscr{D}(e_j)$ with the respective character $\chi^{(\ell)}_r\chi_s$, as in \eqref{iso_char}. So, for instance, to denote the entry $(1,1)$ of $\phi(h_1+h_2)$ in \eqref{iso_char}, we will write $\phi(h_1+h_2)(\chi_1^{(i)}\chi_1,\chi_1^{(i)}\chi_1)$.

\noindent\textbf{Claim.} For each $\chi'\in\hat{H}_i$, we can find unique $\chi'_{k_1},\ldots,\chi'_{k_{n_j}}\in\hat{\mathcal{H}}$ such that $\phi(m_\ell)(\chi',\chi^{(j)\prime}_r\chi'_{k_r})\ne0$, $r=1,2,\ldots,n_j$.

Indeed, let $\mu\in\hat{H}_i$ and let $\mu'\in\hat{H}_j$ be such that $\phi(m_\ell)(\mu,\mu')\ne0$. Now, if $\mu''\in\hat{H}_j$ is such that $\phi(m_\ell)(\mu,\mu'')\ne0$, then $\mu'$ and $\mu''$ cannot be in the same left coset, unless they coincide, as we will prove now. For each $h\in\mathcal{H}$, $hm_\ell h^{-1}=\mu_\ell(h)m_\ell$. On the other hand, $h\phi(m_\ell)h^{-1}(\mu,\mu')=\mu(h)\phi(m_\ell)(\mu,\mu')\mu'(h^{-1})$, and a similar relation holds for $\mu''$. This implies $\mu'|_{\mathcal{H}}=(\mu_\ell\mu^{-1})|_{\mathcal{H}}=\mu''|_{\mathcal{H}}$. Hence, if $\mu'$ and $\mu''$ are in the same coset of $\hat{\mathcal{H}}$, then $\mu'=\mu''$. So, for each $\mu\in\hat{H}_i$, there exist at most $|H_j|/|\mathcal{H}|$ elements $\mu'$ such that $\phi(m_\ell)(\mu,\mu')\ne0$. Thus,$\phi(m_\ell)$ has at most $|H_i||H_j|/|\mathcal{H}|$ non-zero entries.

This implies that $\phi(m_\ell)$ has exactly $|H_iH_j|$ non-zero entries. Moreover, the non-zero entries of the $\phi(m_1),\ldots,\phi(m_s)$ are disjoint.

Now, let $\ell_1,\ell_2$ with $\mu_{\ell_1}=\mu_{\ell_2}$, and fix $\chi'\in\hat{\mathcal{H}}\subseteq\hat{H}_i$, and let $\chi_1',\chi_2'\in\hat{\mathcal{H}}\subseteq\hat{H}_j$ be such that $\phi(m_{\ell_k})(\chi',\chi'_k)\ne0$, $k=1,2$. For each $h\in\mathcal{H}$, again we have $hm_{\ell_k}h^{-1}=\mu_{\ell_k}(h)m_{\ell_k}$. On the other hand, (recall that $\chi_1^{(i)}=1$, $\chi_1^{(j)}=1$)
$$
(h\phi(m_{\ell_k})h^{-1})(\chi',\chi'_k)=\chi'(h)\chi^{\prime-1}_k(h)\phi(m_{\ell_k})(\chi',\chi'_k).
$$
Hence, $\chi'_1=\chi'\mu_{\ell_k}^{-1}=\chi'_2$. Thus, this implies $m_{\ell_1}=m_{\ell_2}$.
\end{proof}

We also have the following immediate consequence of Corollary \ref{main_cor}:
\begin{Cor}\label{Cor5}
Let $G$ be abelian, and let $i<j$. Then:
\begin{enumerate}
\renewcommand{\labelenumi}{(\roman{enumi})}
\item if $H_i\cap H_j=\{1\}$, then $M_{ij}$ is either $0$ or $\mathbb{F}H_{ij}$, where $H_{ij}\cong H_i\times H_j$.
\item if $H_i\subseteq H_j$ (in particular, $H_i=H_j$), then $M_{ij}$ is a graded $\mathbb{F}H_j$-vector space, and $0\le\dim_{\mathbb{F}H_j}M_{ij}\le|H_i|$.
\end{enumerate}\qed
\end{Cor}

Given a graded $(\mathbb{F}H_1,\mathbb{F}H_2)$-bimodule $M$ with decomposition given by \eqref{decomp}, denote $[M]=[(\chi_1,h_1),\ldots,(\chi_s,h_s)]$, where $h_\ell=\deg m_\ell$. We finish this section showing that the knowledge of those characters determines the isomorphism classes of the bimodules, or certain triangular algebras.

\begin{Lemma}\label{bimodule_iso}
Let $V$ and $V'$ be $G$-graded $(\mathbb{F}H_1,\mathbb{F}H_2)$-bimodules, and denote $[V]=[(\chi_1,h_1),\ldots,(\chi_s,h_s)]$, $[V']=[(\chi_1',h_1'),\ldots,(\chi_{s'}',h_{s'}')]$. Then, as $G$-graded bimodules, $V\cong V'$ if, and only if $s=s'$ and there exists $\sigma\in S_s$ such that
$$
h_\ell\equiv h_{\sigma(\ell)}'\pmod{H_1H_2},\quad\chi_{\ell}=\chi'_{\sigma(\ell)},\quad\forall\ell=1,2,\ldots,s.
$$
\end{Lemma}
\begin{proof}
Write $V=\mathbb{F}H_1m_1\mathbb{F}H_2\oplus\cdots\oplus\mathbb{F}H_1m_s\mathbb{F}H_2$, $V'=\mathbb{F}H_1m_1'\mathbb{F}H_2\oplus\cdots\oplus\mathbb{F}H_1m_{s'}'\mathbb{F}H_2$, as in \eqref{decomp}. Let $\psi:V\to V'$ be a graded isomorphism of bimodules. Then $\psi$ is an isomorphism of $H_1\cap H_2$ representations, where the action is given by $L_h\circ R_{h^{-1}}$. Let $V_\chi=\sum_{\chi_i=\chi}\mathbb{F}H_1m_i\mathbb{F}H_2$. Note that $V_\chi$ is the sum of all subrepresentations of $V$ isomorphic to $\chi$. So, $\psi(V_\chi)=\sum_{\chi'_i=\chi}\mathbb{F}H_1m_i'\mathbb{F}H_2=:V_\chi'$. Moreover, $\psi:V_\chi\to V_\chi'$ is a graded isomorphism of $\mathbb{F}(H_1H_2)$-modules. This proves one implication.

The converse is immediate.
\end{proof}

\begin{Lemma}\label{grad_alg_iso}
Let $V$ and $V'$ be $G$-graded $(\mathbb{F}H_1,\mathbb{F}H_2)$-bimodules, and denote $[V]=[(\chi_1,h_1),\ldots,(\chi_s,h_s)]$, $[V']=[(\chi_1',h_1'),\ldots,(\chi_{s'}',h_{s'}')]$. Then, as $G$-graded algebras,
$$
A:=\left(\begin{array}{cc}
\mathbb{F}H_1&V\\
&\mathbb{F}H_2
\end{array}\right)\cong
\left(\begin{array}{cc}
\mathbb{F}H_1&V'\\
&\mathbb{F}H_2
\end{array}\right)=:A'
$$
if and only if $s=s'$ and there exist $\chi\in\widehat{H_1\cap H_2}$ and $\sigma\in S_s$ such that
$$
h_\ell\equiv h_{\sigma(\ell)}'\pmod{H_1H_2},\quad\chi_{\ell}=\chi\chi'_{\sigma(\ell)},\quad\forall\ell=1,2,\ldots,s.
$$
\end{Lemma}
\begin{proof}
The converse is an elementary computation.

Let $\psi:A\to A'$ be a graded isomorphism of algebras. Then we have an induced graded isomorphism $\bar\psi:A/V\to A'/V'$, moreover, $\bar\psi(\mathbb{F}H_i)=\mathbb{F}H_i$. So, there exists $\chi_i\in\hat{H}_i$ such that $\bar\psi(h_i)=\chi_i(h_i)h_i$, for each $h_i\in H_i$, for $i=1,2$. Write $V'=\mathbb{F}H_1m_1'\mathbb{F}H_2\oplus\cdots\oplus\mathbb{F}H_1m_{s'}'\mathbb{F}H_2$, and consider the map $\varphi:A'\to A'$ given by
$$
\varphi\left(\begin{array}{cc}h_1&h_1'm_i'h_2'\\&h_2\end{array}\right)=
\left(\begin{array}{cc}\chi_1^{-1}(h_1)h_1&\chi^{-1}_1(h_1')\chi^{-1}_2(h_2')h_1'm_i'h_2'\\&\chi_2^{-1}(h_2)h_2\end{array}\right).
$$
It is elementary to prove that $\varphi$ is a graded isomorphism of algebras. Moreover, given $h_i\in H_i$ and $m\in V$, we have
$$
\varphi\psi(h_1mh_2)=h_1\varphi\psi(m)h_2.
$$
Thus, $\varphi\psi:V\to\varphi(V')$ is a graded isomorphism of bimodules. Consider the product of the restriction of the characters, $\chi=\chi_1|_{H_1\cap H_2}\cdot\chi_2|_{H_1\cap H_2}$. Note that, if $[V']=[(\chi_1',h_1'),\ldots,(\chi_{s'}',h_{s'}')]$, then the decomposition of $\varphi(V')$ is $[\varphi(V')]=[(\chi\chi_1',h_1'),\ldots,(\chi\chi_{s'}',h_{s'}')]$. Thus, the result follows from Lemma \ref{bimodule_iso}.
\end{proof}

\section{The isomorphism problem\label{S_isoproblem}}
Assume that the conditions of Theorem \ref{class} hold. Consider a $G$-grading on $I(X)$ and let $\mathcal{E}$ be the associated poset. Denote by $e\triangleleft_ce'$ if $e,e'\in\mathcal{E}$, and $e'$ is a cover for $e$; that is, $e\ne e'$, $e\trianglelefteq e'$, and if $e\trianglelefteq f\trianglelefteq e'$, then either $f=e$ or $f=e'$.

Note that the Jacobson Radical $J=J(I(X))$ coincides with $\sum_{e\trianglelefteq e',e\ne e'}eI(X)e'$. Also, $J^2$ is the sum of all $eI(X)e''$ where there exists a chain of different elements $e\triangleleft e'\triangleleft e''$; $J^3$ is the sum of all $eI(X)e'''$, where $e\triangleleft e'\triangleleft e''\triangleleft e'''$, and so on. Given $e,e'\in\mathcal{E}$, $e\trianglelefteq e'$, we can always construct a chain of elements containing only subsequent covers, starting with $e$, and ending with $e'$; however, we can have two different such chains having different number of elements. Nonetheless, if $e\triangleleft_ce'$, then $e\triangleleft e'$ is the unique chain of subsequent covers linking $e$ and $e'$. Let $J_1=\bigoplus_{e\triangleleft_ce'}eI(X)e'$. The previous discussion is the proof of the following:
\begin{Lemma}\label{lem_j1}
Let $I(X)$ be $G$-graded, and let $J_1$ be as above.
\begin{enumerate}
\renewcommand{\labelenumi}{(\roman{enumi})}
\item The restriction of the natural projection $\pi:J_1\to J/J^2$ is a graded linear isomorphism.
\item $J=J_1\oplus\sum_{m>1}J_1^m$.
\end{enumerate}\qed
\end{Lemma}
It should be noticed that $J_1=\mathrm{Span}\{e_{xy}\mid x\le_{Xc}y\}$, where $\le_X$ is the partial order of $X$.

As a consequence, we obtain
\begin{Cor}
Any $G$-grading on $I(X)$ is completely determined by either one of the following graded spaces:
\begin{enumerate}
\renewcommand{\labelenumi}{(\roman{enumi})}
\item $D=\bigoplus_{e\in\mathcal{E}}eI(X)e$, and $J_1=\bigoplus_{e\triangleleft_ce'}eI(X)e'$, or,
\item $I(X)/J$ and $J/J^2$, or,
\item $I(X)/J^2$.
\end{enumerate}
\end{Cor}
\begin{proof}
We prove (i), the others being consequences. Let $x\in I(X)$ be a homogeneous element. Then $x=d+r$, where $d\in D$ and $r\in J$ are both homogeneous elements.  By Lemma \ref{lem_j1}.(ii), $r$ is a linear combination of products of elements of $J_1$. Hence, the element $x$ is completely determined by the homogeneous elements in $D$ and in $J_1$.
\end{proof}
The advantage of (ii) and (iii) above is an assertion independent of $\mathcal{E}$. We can restate the previous corollary in terms of isomorphisms:
\begin{Cor}\label{map_iso}
Consider two $G$-gradings on $I(X)$, and denote them by $A_1$ and $A_2$, and let $\psi:A_1\to A_2$ be a graded homomorphism of algebras. The following assertions are equivalent:
\begin{enumerate}
\renewcommand{\labelenumi}{(\roman{enumi})}
\item $\psi$ is a graded isomorphism,
\item the induced maps by $\psi$ on $A_1/J(A_1)\to A_2/J(A_2)$ and $J(A_1)/J(A_1)^2\to J(A_2)/J(A_2)^2$ are graded isomorphisms,
\item the induced map by $\psi$ on $A_1/J(A_1)^2\to A_2/J(A_2)^2$ is a graded isomorphism.
\end{enumerate}\qed
\end{Cor}

Now, we will prove a necessary condition so that two gradings on $I(X)$ are isomorphic:
\begin{Prop}\label{prop_iso}
Let $I(X)$ be endowed with two $G$-gradings, say $A$ and $A'$, and assume that
$$
A=\left(\begin{array}{cccc}
\mathbb{F}H_1&M_{12}&\ldots&M_{1t}\\
&\ddots&\ddots&\vdots\\
&&\mathbb{F}H_{t-1}&M_{t-1,t}\\
&&&\mathbb{F}H_t
\end{array}\right), A'=
\left(\begin{array}{cccc}
\mathbb{F}H_1'&M_{12}'&\ldots&M_{1t'}'\\
&\ddots&\ddots&\vdots\\
&&\mathbb{F}H_{t'-1}'&M_{t'-1,t'}'\\
&&&\mathbb{F}H_{t'}'
\end{array}\right),
$$
Let $\mathcal{E}$ and $\mathcal{E}'$ be the respective associated posets (see Remark \ref{remarkassociatedposet}). If $A\cong A'$ then $t=t'$, and there exists an isomorphism of posets $\alpha:\mathcal{E}\to\mathcal{E}'$ such that $H_i=H_{\alpha(i)}$, for each $i=1,2,\ldots,t$
\end{Prop}
\begin{proof}
Let $\psi:A\to A'$ be a graded isomorphism. Then we have an induced graded isomorphism $\psi:A/J(A)\to A'/J(A')$. So, $t=t'$ and $\mathbb{F}H_i\cong\mathbb{F}H_{\alpha(i)}$, for all $i=1,2,\ldots,t$, and some permutation $i\mapsto\alpha(i)$. But $\mathbb{F}H_i\cong\mathbb{F}H_{\alpha(i)}$ as graded algebras implies $H_i=H_{\alpha(i)}$.

Since $\psi(J(A))=J(A')$ and $\psi(J(A)^2)=J(A')^2$, we have an induced graded isomorphism $\bar\psi:J(A)/J(A)^2\to J(A')/J(A')^2$. Also, given a cover $e,e'\in\mathcal{E}$, that is, $e\trianglelefteq_c e'$, we have $\psi(eAe')\ne0$. Moreover, $\bar\psi(eAe'+J(A)^2)=\alpha(e)A'\alpha(e')$. Hence, $e\trianglelefteq e'$ implies $\alpha(e)\trianglelefteq'\alpha(e')$. Using a symmetric argument with $\alpha^{-1}$, we conclude that $\alpha$ is an isomorphism of posets.
\end{proof}
\begin{Example}
Let $UT_n$ be $G$-graded. In this case, $UT_n=I(X)$, where $X$ is a chain of $n$ elements. We saw that $\mathcal{E}$ is constructed as a set of subsets of $X$; and every element of $\mathcal{E}$ contains pairwise non-comparable elements of $X$. Since $X$ is a chain, it implies that every element of $\mathcal{E}$ is a singleton, that is, $\{x\}$, for some $x\in X$. Thus, $\mathcal{E}\simeq X$. Hence, $I(X)/J(I(X))$ always has the trivial grading. This fact (with no restrictions on the ground field) was originally proved by Valenti and Zaicev in 2007 \cite{VaZa2007}.
\end{Example}

\begin{Example}
Let $\Gamma_1$ and $\Gamma_2$ be two $G$-gradings on $UT_n$, and name $A_1=(UT_n,\Gamma_1)$ and $A_2=(UT_n,\Gamma_2)$. By the previous example, $\mathcal{E}_1\simeq\mathcal{E}_2\simeq X$ is a chain of $n$ elements. By Proposition \ref{prop_iso}, $\Gamma_1\cong\Gamma_2$ only if $e_{ii}A_1 e_{i+1,i+1}\cong e_{ii}A_2 e_{i+1,i+1}$, for all $i=1,2,\ldots,n-1$. Equivalently, $\Gamma_1\cong\Gamma_2$ only if $\deg_{\Gamma_1}e_{i,i+1}=\deg_{\Gamma_2}e_{i,i+1}$ for all $i=1,2,\ldots,n-1$. The converse is immediate. This fact (indeed, a stronger statement) was originally proved by Di Vincenzo et al in 2004 \cite{VinKoVa2004}.
\end{Example}

\section{Abelian grading group\label{S_abelian}}
In this section, we investigate properties of the gradings on $I(X)$, if the grading group is abelian. Let $G$ be an abelian group. Consider any $G$-grading on $I(X)$, where $X$ is a finite poset, written as Theorem \ref{class}. For any $m\in I(X)$, we denote $s(m)=\{(u,v)\in X^2\mid m(u,v)\ne0\}$. We say that $m$ is \emph{$G$-primitive} if $m$ is homogeneous and there is no homogeneous $y\in I(X)$ with $s(y)\subsetneq s(x)$.

Let $\bar{\mathbb{F}}$ be an algebraically closure of $\mathbb{F}$, and consider $\bar I(X)=I(X)\otimes_\mathbb{F}\bar{\mathbb{F}}$ with the induced $G$-grading. In this case, if the characteristic of $\mathbb{F}$ is adequate, it is well-known that we have a duality between $\hat G$-action and $G$-gradings (see, for instance, \cite[\S 1.4]{EldKoc}). The automorphism group of $I(X)$ is $\mathrm{Aut}(I(X))=\mathrm{Int}(I(X))\cdot\mathrm{mult}(I(X))\cdot\mathrm{Aut}(X)$ (see, for instance, \cite[Theorem 7.3.6]{SpieDon}), where $\mathrm{Int}(I(X))$ is the group of inner automorphisms of $I(X)$, $\mathrm{mult}(I(X))$ is the group of the so-called multiplicative automorphisms of $I(X)$ and $\mathrm{Aut}(X)$ are the induced automorphisms from the automorphisms of the poset $X$.

\begin{Lemma}
If $I(X)$ is expressed as in Theorem \ref{class}, then $\hat G\subset\mathrm{mult}(I(X))\mathrm{Aut}(X)$.
\end{Lemma}
\begin{proof}
Let $f\in\hat G$, and let $x\in X$. Then there is unique minimal idempotent $e$ such that $x\in\mathscr{D}(e)$. Thus, denoting by $D=eI(X)e=\mathrm{Span}\{e_{yy}\mid y\in\mathscr{D}(e)\}$, we have $e_{xx}\in D$. Hence, $e_{xx}$ is a linear combination of homogeneous element in $D$. Since $f(D)=D$, $f(e_{xx})$ is still a linear combination of homogeneous elements in $D$. However, since $f$ is a composition of an inner automorphism, multiplicative automorphism, and an induced automorphism of the poset, actually we have $f(e_{xx})=e_{yy}$, for some $y\in\mathscr{D}(e)$.

Since this construction works for all $x\in X$ and $f$ is bijective, define $\alpha_f:X\to X$ as the (automorphism) $\alpha_f(x)=y$. Denote the induced automorphism on $I(X)$ by the same name. Then, clearly $f\circ\alpha_f^{-1}(e_{xx})=e_{xx}$, for all $x\in X$. By \cite[Proposition 7.3.2]{SpieDon}, $f\circ\alpha_f^{-1}\in\mathrm{mult}(I(X))$.
\end{proof}

Let $\mathscr{A}=\{\alpha_f\mid f\in\hat G\}$ as above, and let $e_1,\ldots,e_t$ be the minimal idempotents. Let $x\in\mathscr{D}(e_i)$, then
\[
\sum_{f\in\hat G}f(e_{xx})=\sum_{f\in\hat G}e_{\alpha_f(x),\alpha_f(x)}=\lambda_i e_i,
\]
is homogeneous, where $\lambda_i\in\mathbb{Z}_{>0}$.

Now, the above elements are homogeneous in $\bar I(X)$, but they also belong to $I(X)$. Hence, they are homogeneous in $I(X)$ as well. Moreover, the automorphisms induced from $\mathrm{Aut}(X)$ are independent on the base field. Thus, we just proved
\begin{Thm}\label{commutativecase}
	Let $\mathbb{F}$ be any field of characteristic zero, $G$ an abelian group and consider any $G$-grading on $I(X)$. Up to a graded isomorphism, let $e_1,e_2,\ldots,e_t$ be the unity of each $\mathbb{F}H_i$, as in Theorem \ref{class}. Then there exists a subset $\mathscr{A}\subset\mathrm{Aut}(X)$ such that $\lambda_i e_i=\sum_{\alpha\in\mathscr{A}}\alpha(e_{xx})$, for any $x\in\mathscr{D}(e_i)$, for some $\lambda_i\in\mathbb{Z}_{>0}$, and for all $i=1,2,\ldots,t$.\qed
\end{Thm}

We conjecture that this theorem is true, even when the grading group is not abelian.

Now, using the results of Sections \ref{S_bivectorspace} and \ref{S_isoproblem}, we are able to provide a classification of isomorphism classes of group gradings on the incidence algebras, given that the grading group is abelian.

\begin{Thm}\label{Thm_iso}
Let $G$ be an abelian group, $X$ a finite poset, and assume two $G$-gradings on $I(X)$, namely $A$ and $A'$. Write
$$
A=\left(\begin{array}{cccc}
\mathbb{F}H_1&M_{12}&\ldots&M_{1t}\\
&\ddots&\ddots&\vdots\\
&&\mathbb{F}H_{t-1}&M_{t-1,t}\\
&&&\mathbb{F}H_t
\end{array}\right), A'=
\left(\begin{array}{cccc}
\mathbb{F}H_1'&M_{12}'&\ldots&M_{1t'}'\\
&\ddots&\ddots&\vdots\\
&&\mathbb{F}H_{t'-1}'&M_{t'-1,t'}'\\
&&&\mathbb{F}H_{t'}'
\end{array}\right),
$$
and denote $\mathcal{E}=\{1,2,\ldots,t\}$, $\mathcal{E}'=\{1,2,\ldots,t'\}$. Let $[M_{ij}]=[(\chi^{(ij)}_1,h^{(ij)}_1),\ldots,(\chi^{(ij)}_{s_{ij}},h^{(ij)}_{s_{ij}})]$, $[M_{ij}']=[(\chi^{(ij)\prime}_1,h^{(ij)\prime}_1),\ldots,(\chi^{(ij)\prime}_{s_{ij}'},h^{(ij)\prime}_{s_{ij}'})]$, where $s_{ij},s_{ij}'\ge0$. Then $A\cong A'$, as $G$-graded algebras if and only if $t=t'$ and there exist an isomorphism of posets $\alpha:\mathcal{E}\to\mathcal{E}'$, and characters $\chi_1\in\hat{H}_1,\ldots,\chi_t\in\hat{H}_t$ such that
\begin{enumerate}
\item $H_i=H_{\alpha(i)}$, for each $i=1,2,\ldots,t$,
\item Given $i\trianglelefteq j$, let $\mathcal{H}_{ij}=H_i\cap H_j$. Then we have $s_{ij}=s_{\alpha(i),\alpha(j)}'$, and there is $\sigma_{ij}\in S_{s_{ij}}$ satisfying (for each $\ell=1,2,\ldots,s_{ij}$)
$$
h^{(ij)}_\ell\equiv h^{(\alpha(i),\alpha(j))\prime}_{\sigma(\ell)}\pmod{H_iH_j},\quad \chi^{(ij)}_\ell=\chi_i|_{\mathcal{H}_{ij}}\cdot\chi_j|_{\mathcal{H}_{ij}}\cdot\chi^{(\alpha(i),\alpha(j))\prime}_{\sigma_{ij}(\ell)}.
$$
\end{enumerate}
\end{Thm}
\begin{proof}
The 'only if' part follows from Proposition \ref{prop_iso} and the proof of Lemma \ref{grad_alg_iso}. To prove the 'if' part, we notice that the isomorphism of posets $\alpha$ induces an algebra isomorphism. Then, we apply Corollary \ref{map_iso} and the proof of Lemma \ref{grad_alg_iso}.
\end{proof}

\section{Examples\label{S_Examples}}
\subsection{It is not always possible to reduce to cyclic groups.} Consider a commutative graded division algebra $D$. It is natural to ask if we can reduce such graded algebra to a direct sum of graded subalgebras, where each of them are graded by cyclic groups. The answer is no, as shown by the following example. Let $X=\{1,2,3,4\}$ where $i\le j$ if and only if $i=j$. We can define a $G=\mathbb{Z}_2\times\mathbb{Z}_2$ grading on $I(X)$ posing:
\begin{eqnarray*}
\deg e_{11}+e_{22}+e_{33}+e_{44}&=&(0,0),\\
\deg e_{11}-e_{22}-e_{33}+e_{44}&=&(1,0),\\
\deg e_{11}+e_{22}-e_{33}-e_{44}&=&(0,1),\\
\deg e_{11}-e_{22}+e_{33}-e_{44}&=&(1,1).
\end{eqnarray*}
In particular, there is no homogeneous idempotent $e'\in I(X)$ with $s(e')\subset s(e)$ other than $e$ itself. Hence we can not break $I(X)$ as a sum of two graded division algebras.\\

%
%
\subsection{Non-existence of homogeneous multiplicative basis.} Let $G=\mathbb{Z}_2\times\mathbb{Z}_2\times\mathbb{Z}$ and $X$ be the poset whose Hasse diagram is the following:

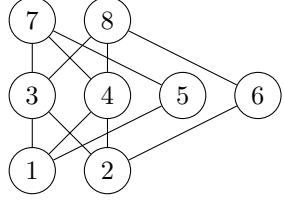
\begin{center}
\begin{tikzpicture}
\node[circle,draw] (a) at (0,0) {$1$};
\node[circle,draw] (b) at (1,0) {$2$};
\node[circle,draw] (c) at (0,1) {$3$};
\node[circle,draw] (d) at (1,1) {$4$};
\node[circle,draw] (g) at (0,2) {$7$};
\node[circle,draw] (h) at (1,2) {$8$};
\node[circle,draw] (e) at (2,1) {$5$};
\node[circle,draw] (f) at (3,1) {$6$};
\draw (a) -- (c) -- (b) -- (d) -- (a);
\draw (c) -- (g) -- (d) -- (h) -- (c);
\draw (a) -- (e) -- (g);
\draw (b) -- (f) -- (h);
\end{tikzpicture}
\end{center}

Define the following diagonal elements to be homogeneous:
\begin{eqnarray*}
\deg e_{11}+e_{22}&=&(0,0,0),\\
\deg e_{11}-e_{22}&=&(1,0,0),\\
\deg e_{33}+e_{44}&=&(0,0,0),\\
\deg e_{33}-e_{44}&=&(0,1,0),\\
\deg e_{55}+e_{66}&=&(0,0,0),\\
\deg e_{55}-e_{66}&=&(1,0,0),\\
\deg e_{77}+e_{88}&=&(0,0,0),\\
\deg e_{77}-e_{88}&=&(1,0,0),
\end{eqnarray*}
and, the following elements of the Jacobson radical to be homogeneous:
\begin{eqnarray*}
\deg e_{13}+e_{14}+e_{23}+e_{24}&=&(0,0,1),\\
\deg e_{13}+e_{14}-e_{23}-e_{24}&=&(1,0,1),\\
\deg e_{13}-e_{14}+e_{23}-e_{24}&=&(0,1,1),\\
\deg e_{13}-e_{14}-e_{23}+e_{24}&=&(1,1,1),
\end{eqnarray*}
\begin{eqnarray*}
\deg e_{37}+e_{38}+e_{47}+e_{48}&=&(0,0,1),\\
\deg e_{37}-e_{38}+e_{47}-e_{48}&=&(1,0,1),\\
\deg e_{37}+e_{38}-e_{47}-e_{48}&=&(0,1,1),\\
\deg e_{37}-e_{38}-e_{47}+e_{48}&=&(1,1,1),
\end{eqnarray*}
\begin{eqnarray*}
\deg e_{15}+e_{26}&=&(0,0,1),\\
\deg e_{15}-e_{26}&=&(1,0,1),\\
\deg e_{57}+e_{68}&=&(0,0,1),\\
\deg e_{57}-e_{68}&=&(1,0,1),
\end{eqnarray*}
\begin{eqnarray*}
\deg e_{17}+e_{28}&=&(0,0,2),\\
\deg e_{17}-e_{28}&=&(1,0,2),\\
\deg e_{27}+e_{18}&=&(0,0,2),\\
\deg e_{27}-e_{18}&=&(1,0,2).
\end{eqnarray*}
Let $G_1=\mathbb{Z}_2\times0\times0$ and $G_2=0\times\mathbb{Z}_2\times0$ be subgroups of $G$. The induced poset, as in Lemma \ref{ordering}, is

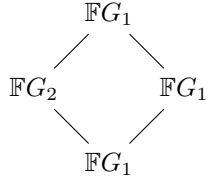
\begin{center}
\begin{tikzpicture}
\node (a) at (1,0) {$\mathbb{F}G_1$};
\node (b) at (0,1) {$\mathbb{F}G_2$};
\node (c) at (2,1) {$\mathbb{F}G_1$};
\node (d) at (1,2) {$\mathbb{F}G_1$};
\draw (a) -- (b) -- (d) -- (c) -- (a);
\end{tikzpicture}\\
\end{center}
where, by abuse of notation, every $\mathbb{F}G_i$ represents its homogeneous idempotent generator. Moreover, with the notation of Theorem \ref{class}, the above grading is isomorphic to
$$
I(X)\cong\left(\begin{array}{cccc}
\mathbb{F}G_1&\mathbb{F}(G_1\times G_2)&\mathbb{F}G_1&\mathbb{F}G_1\oplus \mathbb{F}G_1\\
&\mathbb{F}G_2&0&\mathbb{F}(G_2\times G_1)\\
&&\mathbb{F}G_1&\mathbb{F}G_1\\
&&&\mathbb{F}G_1
\end{array}\right).
$$
We claim that $I(X)$ has no multiplicative homogeneous basis. Let us use the notation of Theorem \ref{class}, that is, let $M_{12}=\mathbb{F}(G_1\times G_2)$, $M_{13}=\mathbb{F}G_1$, $M_{24}=\mathbb{F}(G_2\times G_1)$, $M_{34}=\mathbb{F}G_1$. The main idea is that each nonzero element obtained by the product of homogeneous elements $M_{12}M_{24}$ must have 4 nonzero entries; while the nonzero elements of the product $M_{13}M_{34}$ have only 2 nonzero entries. Thus, we cannot construct a homogeneous basis of this algebra.

The proof is tedious, and we list the main steps. Assume that $\mathcal{B}$ is a homogeneous basis of $I(X)$.
\begin{enumerate}
\renewcommand{\labelenumi}{(\arabic{enumi})}
\item $\mathcal{B}$ consists of a homogeneous basis of the diagonal $\mathbb{F}G_1\oplus\mathbb{F}G_2\oplus\mathbb{F}G_1\oplus\mathbb{F}G_1$, the space $M_{12}\oplus M_{13}\oplus M_{24}\oplus M_{34}$, and $M_{14}$.
\item $\mathcal{B}$ contains a homogeneous basis of $\mathbb{F}G_2$.
\item So, $\mathcal{B}$ contains a homogeneous basis of $M_{12}$ and $M_{24}$ as well. Hence, the elements of $\mathcal{B}$ that belongs to $M_{14}$ must have four non-zero entries.
\item For $i<j$, denote by $\pi_{ij}:I(X)\to M_{ij}$ the projection. Given $v\in\mathcal{B}$, then $\pi_{13}(v)\ne0$ implies $\pi_{12}(v)=0$. Similarly, $\pi_{34}(v)\ne0$ implies $\pi_{24}(v)=0$.
\item So, we can find $v_1,v_2\in\mathcal{B}$ such that $\pi_{13}(v_1)\ne0$, $\pi_{34}(v_2)\ne0$ and $v_1v_2\ne0$. However, $v_1v_2\in M_{14}$ contains only two non-zero entries, and it should be a multiple of an element of $\mathcal{B}$. This is a contradiction.
\end{enumerate}

\subsection{Non-abelian grading.} Fix a $G$-grading on $I(X)$. If we can find a commutative group $H$, and a $H$-grading on $I(X)$, equivalent to the original $G$-grading (in the sense of \cite[Definition 1.14, p. 14]{EldKoc}), then we can utilize Theorem \ref{commutativecase}, and obtain the same statement for non-abelian grading group. But the following example shows that we can not always do it. Let $G=S_3$ be the group of permutations of 3 elements. Let $X=\{1,2,3,4,5,6\}$, where $i\le i+3$, for all $i$. In Hasse diagram:
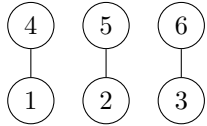
\begin{center}
\begin{tikzpicture}
\node[circle,draw] (a) at (0,0) {$1$};
\node[circle,draw] (b) at (1,0) {$2$};
\node[circle,draw] (c) at (2,0) {$3$};
\node[circle,draw] (d) at (0,1) {$4$};
\node[circle,draw] (e) at (1,1) {$5$};
\node[circle,draw] (f) at (2,1) {$6$};
\draw (a) -- (d);
\draw (b) -- (e);
\draw (c) -- (f);
\end{tikzpicture}\\
\end{center}

Let $\omega\ne1$ be such that $\omega^3=1$. Define the following elements to be homogeneous:
\begin{eqnarray*}
\deg e_{11}+e_{22}+e_{33}&=&(1),\\
\deg e_{11}+\omega e_{22}+\omega^2e_{33}&=&(123),\\
\deg e_{11}+\omega^2e_{22}+\omega e_{33}&=&(132),\\
\deg e_{14}+e_{25}+e_{36}&=&(12),\\
\deg e_{14}+\omega e_{25}+\omega^2e_{36}&=&(13),\\
\deg e_{14}+\omega^2e_{25}+\omega e_{36}&=&(23),\\
\deg e_{44}+e_{55}+e_{66}&=&(1),\\
\deg e_{44}+\omega e_{55}+\omega^2e_{66}&=&(132),\\
\deg e_{44}+\omega^2e_{55}+\omega e_{66}&=&(123).
\end{eqnarray*}
This indeed well defines a $S_3$-grading on $I(X)$. The key point is the product
$$
(e_{11}+\omega e_{22}+\omega^2e_{33})(e_{14}+e_{25}+e_{36})(e_{44}+\omega^2e_{55}+\omega e_{66})=e_{14}+e_{25}+e_{36}
$$
Thus, if $g=\deg e_{11}+\omega e_{22}+\omega^2e_{33}$, and $h=\deg e_{14}+e_{25}+e_{36}$, we obtain the equation $ghg=h$. So, the grading cannot be equivalent to an abelian grading, unless $g^2=1$, which is never the case. Hence, the extension of Theorem \ref{commutativecase} to the non-abelian case, if possible, is not so direct.

\subsection{Reconstruction of algebras}
A natural question is which algebras can be realized as an incidence algebra endowed with a grading.

In general, this is a hard question and we do not have a complete answer. It seems that we cannot obtain a general answer for all kinds of finite posets.

Every group algebra can be realized as an incidence algebra with a grading, by Lemma \ref{field}. If the associated poset $\mathcal{E}$ has 2 elements, then we can also realize:
\begin{Prop}
Let $G$ be an abelian group, and assume that $\mathbb{F}$ contains enough roots of unit. Let $H_1,H_2$ be finite subgroups of $G$, and $M$ a graded $(\mathbb{F}H_1,\mathbb{F}H_2)$-bimodule, satisfying $[M]=[(\chi_1,h_1),\ldots,(\chi_m,h_m)]$ with $\chi_i\ne\chi_j$ for $i\ne j$. Then there exists a finite poset $X$ and a $G$-grading on $I(X)$ such that
$$
I(X)\cong\left(\begin{array}{cc}\mathbb{F}H_1&M\\&\mathbb{F}H_2\end{array}\right).
$$
\end{Prop}
\begin{proof}
Let $X=X_1\cup X_2$ be a poset, where $X_1$ and $X_2$ are disjoint sets, and $X_i$ contains $|H_i|$ elements. We construct an ordering $\le_X$ on $X$ where $i\le_X j$ implies $i\in X_1$, and $j\in X_2$. Thus
$$
I(X)\cong\left(\begin{array}{cc}\mathbb{F}^{|H_1|}&V\\0&\mathbb{F}^{|H_2|}\end{array}\right).
$$

Since, as ordinary algebras, $\mathbb{F}H_i\cong\mathbb{F}^{|H_i|}$, we can endow in $\mathbb{F}^{|H_i|}$ a $H_i$-grading, in such a way that $\mathbb{F}^{|H_i|}\cong\mathbb{F}H_i$ as graded algebras. Recall that, if $e_i\in\mathbb{F}H_i$ is the homogeneous idempotent, then $\ell(e_1,e_2)$ is the linking number of $e_1$ and $e_2$. We shall construct $\le_X$ to have $\ell(e_1,e_2)=\dim_\mathbb{F}M$.

Write $M\cong\mathbb{F}Hm_1\oplus\ldots\oplus\mathbb{F}Hm_s$, where $0\le s\le\frac{|H_1||H_2|}{|H|}$. Now, we construct exactly $s|H|$ relations $i\le_X j$, in such a way that $\ell(i,e_2)$, $i\in X_1$, and $\ell(e_1,j)$, $j\in X_2$, are constants. As a graded bimodule, $V$ is generated by exactly $s$ elements, say $r_1,\ldots,r_s$. Defining $\deg r_k=\deg m_k$, for $k=1,2,\ldots,s$, we obtain the desired realization. We left to the reader to fil all the details of this construction.
\end{proof}

If $|\mathcal{E}|>2$, then we should deal with compatibility conditions, as the following example suggests.
\begin{Example}
Let $G$ be an abelian group containg $H\cong\mathbb{Z}_2$. Assume that $I(X)$ is a $G$-graded incidence algebra such that
$$
I(X)\cong\left(\begin{array}{ccc}
\mathbb{F}H&M_{12}&M_{13}\\
&\mathbb{F}H&M_{23}\\
&&\mathbb{F}H
\end{array}\right),
$$
where $\dim_{\mathbb{F}H}M_{ij}=2$, for all $1\le i<j\le 3$. Write $M_{ij}=\mathbb{F}Hm_{ij}^{(1)}\oplus\mathbb{F}H m_{ij}^{(2)}$, where each $m_{ij}^{(k)}$ is homogeneous of degree $h_{ij}^{(k)}$. Then $I(X)$ is consistent if, and only if, the following conditions hold:
\begin{align*}
h_{12}^{(1)}h_{23}^{(1)}=h_{13}^{(1)}=h_{12}^{(2)}h_{23}^{(2)},\\
h_{12}^{(2)}h_{23}^{(1)}=h_{13}^{(2)}=h_{12}^{(1)}h_{23}^{(2)}.
\end{align*}
\end{Example}

\begin{Question}
Determine which graded algebras can be realized as an incidence algebra endowed with a $G$-grading.
\end{Question}

\section*{Acknowledgment}
This work started when the third author visited the State University of Maring\'a. The third author would like to thank all the hospitality and cordiality received from Professor Ednei and Jonathan, and the resources of the State University of Maring\'a. All the authors are grateful to Dr.~M.~Kochetov for providing a proof of Proposition \ref{kochetov}.

\end{document}